\documentclass[12pt,oneside]{amsart}

\usepackage{amssymb,latexsym}

\usepackage{hyperref}
\hypersetup{
  colorlinks   = true, 
  urlcolor     = blue, 
  linkcolor    = blue, 
  citecolor   = red 
}
\usepackage{anysize}
\marginsize{2cm}{2cm}{2cm}{2cm}

\usepackage[pdftex]{graphicx}
\usepackage{subcaption}
\usepackage{mathtools}

\usepackage{pdftexcmds}
\usepackage{catchfile}

\newtheorem{theorem}{Theorem}[section]

\newtheorem{lemma}[theorem]{Lemma}

\newtheorem{corollary}[theorem]{Corollary}

\theoremstyle{definition}
\newtheorem{definition}[theorem]{Definition}

\theoremstyle{remark}
\newtheorem{remark}[theorem]{Remark}
\newtheorem{question}[theorem]{Question}

\numberwithin{equation}{section}

\DeclareMathOperator{\id}{id}
\DeclareMathOperator{\SE}{SE}

\DeclareMathOperator{\length}{length}
\DeclareMathOperator{\graph}{graph}
\DeclareMathOperator{\Ham}{Ham}

\DeclareMathOperator{\Symp}{Symp}
\DeclareMathOperator{\Barc}{Bar}
\DeclareMathOperator{\Length}{Length}
\DeclareMathOperator{\SO}{SO}

\newcommand{\newsup}{\mathop{\smash{\mathrm{sup}}}}
\newcommand{\BM}{\mathcal{B}\mathcal{M}}

\renewcommand{\epsilon}{\varepsilon}
\renewcommand{\phi}{\varphi}

\begin{document}

\title{Billiards and Hofer's geometry}

\author{Mark Berezovik}

\address{Mark Berezovik, School of Mathematical Sciences, Tel Aviv University, Israel,  69978}

\email{m.berezovik@gmail.com}

\author{Konstantin Kliakhandler}

\address{Konstantin Kliakhandler}

\email{kosta@slumpy.org}

\author{Yaron Ostrover}

\address{Yaron Ostrover, School of Mathematical Sciences, Tel Aviv University, Israel,  69978}

\email{ostrover@tauex.tau.ac.il}

\author{Leonid Polterovich}

\address{Leonid Polterovich, School of Mathematical Sciences, Tel Aviv University, Israel,  69978}

\email{polterov@tauex.tau.ac.il}

\thanks{Mark Berezovik and Yaron Ostrover were partially supported by the ISF grant 938/22} 
\thanks{Mark Berezovik and Leonid Polterovich were partially supported by the ISF-NSFC grant 3231/23}


\begin{abstract}
We present a link between billiards in convex plane domains and Hofer's geometry, an area of symplectic topology. For smooth strictly convex billiard tables, we prove that the Hofer distance between the corresponding billiard ball maps admits an upper bound in terms of a simple geometric distance between the tables. We use this result to show that the billiard ball map of a convex polygon lies in the completion, with respect to Hofer's metric, of the group of smooth area-preserving maps of the annulus. Finally, we discuss related connections to dynamics and pose several open problems.
\end{abstract}

\maketitle

\section{Introduction} \label{sec:intro}

In the present paper we discuss a link between billiards in convex plane domains and Hofer's geometry, an area of modern symplectic topology. Recall that the dynamics of a billiard ball moving inside a convex table $K$ with a smooth boundary of positive curvature, and obeying the law that the angle of incidence equals the angle of reflection, is modeled by a smooth map $\psi_K$ of the open annulus $A = S^1 \times (-1,1)$ that preserves area and orientation and extends continuously to the identity on the boundary. Denote by $\Symp(A,\partial A)$ the group of such diffeomorphisms, and by $\Ham (A,\partial A)$ its subgroup formed by Hamiltonian diffeomorphisms of $A$ generated by the Hamiltonians that continuously extend by zero to the boundary (see Definition~\ref{def:ham_path} below).  Let $\psi_D$ be the billiard ball map of a disc. As we shall see below, billiard ball maps lie in the subset
\begin{equation}\label{def:BM}
\BM:= \psi_D \cdot \Ham (A,\partial A) \subset  \Symp(A,\partial A)\;.
\end{equation}
For any two elements $\phi, \psi \in \BM$, the diffeomorphism $\phi^{-1}\psi$ lies in  $\text{Ham} (A,\partial A)$. Thus, the set $\BM$ is equipped with the Hofer metric $d_H$ \cite{hofer_1990}, which, roughly speaking, measures the minimal amount of energy required to generate a path connecting $\phi$ and $\psi$.

Our first result states that the Hofer distance $d_H(\psi_K, \psi_L)$ between the billiard ball maps associated with the tables $K$ and $L$ admits an upper bound in terms of a simple geometric distance between $K$ and $L$ (see Theorem~\ref{thm:main} below). We discuss
an application of this result to the billiard ball dynamics in Corollary \ref{cor-dyn}.

Next, we focus on the case when $K$ is a convex polygon. This table is not strictly convex, and the boundary is not smooth. We prove that one can approximate $K$ by a sequence of smooth convex tables with positive curvature so that the corresponding sequence of billiard ball maps is Cauchy in the Hofer metric (see Theorem~\ref{thm:approx} below). Roughly speaking, this means that the billiard ball map of a polygon lies in the completion of $\BM$ with respect to Hofer's metric. While this and related completions~\cite{arnaud2024higher, buhovsky2024dichotomy, humiliere2008some, viterbo2022supports} have recently attracted considerable attention, dynamically and physically meaningful examples remain quite rare.

The paper is organized as follows. In Section~\ref{sec:preliminaries}, we set the stage by defining the Hofer metric and introducing the geometric distance between billiard tables, which will be further discussed in Section~\ref{sec:comparison}. In Section~\ref{sec:main_theorem} we formulate and prove the comparison theorem (Theorem~\ref{thm:main}) between the Hofer distance and the geometric distance, as well as present an application to dynamics.
Section~\ref{sec:approximation_of_polygons} deals with the approximation of polygons (Theorem~\ref{thm:approx}). Finally, in Section~\ref{sec:discussion} we discuss some future directions and open problems.

\subsection*{Acknowledgments} The authors thank Misha Bialy for fruitful discussions, and Jinxin Xue and Michael Fraiman for helpful comments. The authors acknowledge the use of ChatGPT 6 Astra (OpenAI) for checking the accuracy of the manuscript and suggesting a direct approximation argument in Remark~\ref{rem:approx_smooth}. Part of this work was carried out during L.P.'s sabbatical stay at the University of Chicago, whose stimulating research environment is gratefully acknowledged.

\section{Preliminaries} \label{sec:preliminaries}

Throughout the paper we shall deal with several closely related notions of billiard tables.

\begin{definition}
A \emph{billiard table} $K \subset \mathbb{R}^2$ is a smooth convex body whose boundary $\partial K$ is of length $1$ and has positive curvature. A \emph{marked billiard table} is a pair $(K, S)$, where $K$ is a billiard table, and $S$ is a point on the boundary $\partial K$.
\end{definition}

\noindent
The set of all marked billiard tables will be denoted by $\mathcal{B}$. Note that the group $S^1= \mathbb{R}/Z$ acts freely on $\mathcal{B}$ by counterclockwise rotations of the marked point along the boundary. In this notation, the space of (non-marked) billiard tables is $\mathcal{B}/S^1$. Additionally, the group $\SE(2)$ of orientation-preserving Euclidean isometries of the plane acts on $\mathcal{B}$. We denote $\mathcal{AB} = \mathcal{B}/\SE(2)$.

\medskip\noindent{\bf Convention:} For brevity, we call elements of $\mathcal{B}/S^1$, $\mathcal{B}$, and $\mathcal{AB}$ simply \emph{ billiard tables}; the precise meaning will be clear from the context.

\medskip

Consider a billiard table $(K,S)$ and let $\gamma \colon S^1 = \mathbb{R}/\mathbb{Z} \to \mathbb{R}^2$ be the counterclockwise arc-length parametrization of $\partial K$ with $\gamma(0) = S$.
Such a curve is uniquely determined by the billiard table, and since it uniquely determines the table itself, we shall refer to it as a billiard table as well.

\begin{definition}
	Let $\alpha,\beta$ be two billiard tables. We define the ``geometric'' distance between them as
	\[
		d_{B}(\alpha,\beta) = \inf [l_{B}(\gamma_s)],
	\]
	where the infimum is taken over all smooth homotopies $\gamma_s$ between $\alpha$ and $\beta$ such that $\gamma_s \in \mathcal{B}$ for every $s \in [0,1]$. The length of a homotopy is defined as follows
	\[
		l_{B}(\gamma_s) = \int_0^1 \max_{q \in S^1} \left\|\frac{\partial \gamma_s}{\partial s} (q)\right\| \, ds,
	\]
	where $\|\cdot\|$ is the Euclidean norm.
\end{definition}

\begin{remark}
	The distance $d_{B}$ is the \emph{length metric} induced by the $C^0$-distance on $\mathcal{B}$. See~\cite{gromov2007metric} for the definition of a metric length structure.
\end{remark}

The metric $d_B$ descends to a metric on $\mathcal{AB}$ which we denote by $d_{AB}$\footnote{The reader can check that $d_{AB}$ is non-degenerate directly, but this also follows from Remark~\ref{rem:nondegeneracy}, Theorem~\ref{thm:main} and Lemma~\ref{lem:inject}.}.

\medskip

Let $A$ be the open annulus $S^1 \times (-1,1)$ and $\bar{A}$ be the closed annulus $S^1 \times [-1,1]$.

\begin{definition}
	Every curve $\gamma \in \mathcal{B}$ defines a map $\psi_\gamma \colon A \to A$,\, $\psi_\gamma(q,p)=(Q,P)$ so that
	\begin{align*}
	 p &= \left\langle \frac{\gamma(Q) - \gamma(q)}{\|\gamma(Q) - \gamma(q)\|}, \gamma'(q)\right\rangle,\\
	 P &= \left\langle \frac{\gamma(Q) - \gamma(q)}{\|\gamma(Q) - \gamma(q)\|}, \gamma'(Q)\right\rangle.
	\end{align*}
	Such a map is unique and is called the \emph{billiard ball map} of the billiard table $\gamma$ (see~\cite{tabachnikov2005geometry,siburg2004principle}).
\end{definition}

	The map $\psi_{\gamma}$ is smooth by the implicit function theorem, and it is well known that this is a symplectomorphism with respect to the standard symplectic structure $\omega = dp \wedge dq$ on~$A$~\mbox{\cite[Section 8.3]{Mcduff_salamon}}. Note that the map $\psi_\gamma$ can be continuously extended to $\bar{A}$ by setting $\psi_\gamma|_{\partial A} = \id_{\partial A}$, but this map would not be smooth near the boundary~\cite[Section 3]{Marvizi}.

\begin{definition}\label{def:ham_path}
	Let $\psi_s$ be a smooth path of symplectomorphisms of the annulus $A$ for $s \in [0,1]$ that can be continuously extended as a map $A \times [0,1] \to A$  to a map $\bar{A} \times [0,1] \to \bar{A}$ by setting $\psi_s|_{\partial A} = \id_{\partial A}$. We shall call this path \emph{Hamiltonian} if there exists a smooth family of functions $H_s \colon A \to \mathbb{R}$ such that $i_{X_s} \omega = -dH_s$ for every $s \in [0,1]$, where $X_s$ is uniquely determined by
	\[
		\frac{d\psi_{s}}{ds} = X_s \circ \psi_s,
	\]
	and $H_s$ can be continuously  extended as a function on $A \times [0,1]$ to a function on $\bar{A} \times [0,1]$ by setting $H_s|_{\partial A} = 0$. Note that the family of functions $H_s \colon A \to \mathbb{R}$ is unique due to this boundary condition.
\end{definition}

\begin{definition}

	$\Ham (A, \partial A) = \{\psi_1 :\ \psi_s\ \text{--- Hamiltonian path such that}\quad \psi_0 = \id_A\}$.

\end{definition}

\begin{definition}
	Let $\psi_s \colon A \to A$ be a Hamiltonian path and $H_s \colon A \to \mathbb{R}$ be the corresponding family of functions. The \emph{Hofer length}~\cite{hofer_1990} of this path is defined as follows
	\[
		l_H(\psi_s) = \int_0^1 \left(\max_{\bar{A}} H_s - \min_{{\bar{A}}} H_s\right) ds = \int_0^1 \left(\sup_{A} H_s - \inf_{A} H_s\right) ds .
	\]
\end{definition}

\begin{definition}
	Let $\varphi, \psi$ be smooth symplectomorphisms of $A$ that can be continuously extended to the boundary by setting $\varphi|_{\partial A} = \psi|_{\partial A} = \id_{\partial A}$. The \emph{Hofer distance} $d_H$ between $\varphi$ and $\psi$ is defined as the infimum of the lengths of Hamiltonian paths that join them.
\end{definition}

\begin{remark} \label{rem:nondegeneracy}
	To see that $d_H$ in that specific case is non-degenerate let us consider a map $\psi \in \Ham (A, \partial A)$ that is not the identity. Then there exists a point $x \in A$ such that $\psi(x) \neq x$. Let $U$ be a small neighborhood of $x$ such that $\psi(U) \cap U = \varnothing$ and the closure of $U$ is a subset of $A$. Let $\widetilde{e}(U)$ be the standard displacement energy of $U$, that is
	\[
		\widetilde{e}(U) = \inf \left\{ \int_0^1 \left(\sup_{A} \widetilde{H}_s - \inf_{A} \widetilde{H}_s\right) ds \right\},
	\]
	where the infimum is taken over all Hamiltonians $\widetilde{H}_s$ compactly supported in the interior of the annulus, such that the corresponding time-1 map $\phi_{\widetilde{H}}$ satisfies $\phi_{\widetilde{H}}(U) \cap U = \varnothing$. Since $U$ is a nonempty open set, we have $\widetilde{e}(U) > 0$ (see~\cite{lalonde1995geometry} and~\cite[Section 12.3]{Mcduff_salamon}).

	Now, let $\psi_s$ be a Hamiltonian path with the Hamiltonian function $H_s$, such that $\psi_0 = \id_{A}$ and $\psi_1 = \psi$. Let $\rho$ be a smooth bump function with compact support in $A$ such that $\rho$ equals one in a small neighborhood of the trace of $U$ under the flow $\psi_s$. Then the Hamiltonian function $\widetilde{H}_s = \rho \cdot H_s$ generates a time-1 map $\phi_{\widetilde{H}}$ such that $\phi_{\widetilde{H}}(U) = \psi(U)$, hence $\phi_{\widetilde{H}}(U) \cap U = \varnothing$. Then, by the definition of $\widetilde{e}(U)$, we have
	\[
			0 < \widetilde{e}(U) \leq \int_0^1 \left(\sup_{A} \widetilde{H}_s - \inf_{A} \widetilde{H}_s\right) ds \leq	 \int_0^1 \left(\sup_{A} H_s - \inf_{A} H_s\right) ds.
	\]
	Therefore,
	\[
		0< \widetilde{e}(U) \leq d_{H}(\id_A, \psi).
	\]
	Thus, $d_H$ is indeed non-degenerate.
\end{remark}

	In the last remark we considered the notion of the standard displacement energy. Let us introduce the notion of the displacement energy associated with the group $\Ham(A,\partial A)$.

	\begin{definition}
	 Let $X$ be a subset of $A$. Then the \emph{displacement energy} of $X$ associated with the group $\Ham(A,\partial A)$ is defined as follows
	 \[
	 		e(X) = \inf\{d_H(\id_{A},\psi)|\, \psi \in \Ham(A,\partial A), \, \psi(X) \cap X = \varnothing \}.
	 \]
	\end{definition}

	\begin{remark}\label{rem:energy_compare}
		It is easy to see that $e(X) \leq \widetilde{e}(X)$ for every set $X \subseteq A$. If the closure of $X$ is also contained in $A$, then, by the argument in Remark~\ref{rem:nondegeneracy}, one can obtain $e(X) = \widetilde{e}(X)$.   
	\end{remark}

	To sum up, we have a natural map $\Upsilon: \mathcal{B} \to \mathcal{BM}$ (the latter space is defined in \eqref{def:BM}) which associates a billiard ball map with a marked billiard table. The map $\Upsilon$ is invariant under the action of $\SE(2)$ on $\mathcal{B}$. Let us discuss the effect of changing the marked point. Observe that the actions of $\SE(2)$ and $S^1$ on $\mathcal{B}$ commute. Thus, we have a natural isometric action of $S^1$ on $\mathcal{AB}$ given by $r \cdot [\gamma(q)] = [\gamma(q+r)]$, where $r \in S^1$ and $[\gamma(q)] \in \mathcal{AB}$. There is also an $S^1$-action on the annulus $A = S^1 \times (-1,1)$ given by $\varphi_r (q,p) = (q-r,p)$. This allows us to define an isometric action of $S^1$ on $\BM$ given by $r\cdot \psi = \varphi_r \circ \psi \circ \varphi_r^{-1}$. The map $(\mathcal{AB}, d_{AB}) \to (\BM, d_{H})$ is equivariant with respect to the corresponding $S^1$-actions. This map and these $S^1$-actions can be extended to the completions $\overline{\mathcal{AB}}$ and $\overline{\BM}$. The extended map is $S^1$-equivariant.

\section{A comparison theorem} \label{sec:main_theorem}
\begin{theorem}\label{thm:main}
	Any two billiard ball maps can be joined by a Hamiltonian path, and hence they belong to the set $\BM$ defined by \eqref{def:BM}. Furthermore, for any two billiard tables $\alpha$, $\beta$ and corresponding billiard ball maps $\psi_\alpha$, $\psi_\beta$ the following inequality holds:
	\begin{equation}\label{eq:main_lipschitz}
		d_{H}(\psi_{\alpha}, \psi_{\beta}) \leq 4\cdot d_B(\alpha,\beta).
	\end{equation}
\end{theorem}

\begin{remark}
	 In fact, we shall prove a stronger statement. The Hofer length structure induces an intrinsic metric $\widehat d$ on the subset $\Upsilon(\mathcal B)\subset\mathcal{BM}$ consisting of all billiard ball maps. The distance $\widehat d$ between two such maps is the infimum of the Hofer lengths of all smooth paths joining them and consisting entirely of billiard ball maps. Our upper bound~\eqref{eq:main_lipschitz} remains valid with $d_H$ replaced by $\widehat d$.
\end{remark}

The following lemma is needed for the proof of the theorem.

\begin{lemma}\label{lem:ham_gen_fund}
	Let $\gamma_s$ be a smooth path of billiard tables and $\psi_s$ be the corresponding path of symplectomorphisms of the annulus $A$. Consider the smooth family of functions $F_s \colon (S^1 \times S^1) \setminus {\Delta_{S^1}} \to \mathbb{R}$ defined as
	\[
		F_s(q,Q) = \|\gamma_s(q) - \gamma_s(Q) \|,
	\]
	where $\Delta_{S^1} = \{(q,q)\in S^1 \times S^1: q \in S^1\}$. Then the family of functions defined as
	\[
		H_s(Q,P) = - \frac{\partial F_s}{\partial s}(q_s(Q,P),Q)
	\]
	 is the family of Hamiltonian functions for the path $\psi_s$, where
	 \[
	 	(q_s(Q,P), p_s(Q,P)) = \psi^{-1}_s(Q,P).
	 \]
	 In particular, the path $\psi_s$ is Hamiltonian.
\end{lemma}

\begin{proof}
	From the definition of the billiard ball map we have
	\begin{align*}
		\frac{\partial F_s}{\partial q} (q, Q_s(q,p)) &= -p,\\
		\frac{\partial F_s}{\partial Q} (q, Q_s(q,p)) &= P_s(q,p),
	\end{align*}
	where $(Q_s(q,p), P_s(q,p)) = \psi_s(q,p)$. Differentiating these equations with respect to $s$ shows
	\begin{align}
		\frac{\partial^2 F_s}{\partial s\, \partial q} + \frac{\partial^2 F_s}{\partial Q\, \partial q} \cdot \frac{\partial Q_s}{\partial s} &= 0,\label{eq:diff_F_1}\\
		\frac{\partial^2 F_s}{\partial s\, \partial Q}+ \frac{\partial^2 F_s}{\partial Q^2} \cdot \frac{\partial Q_s}{\partial s}	&= \frac{\partial P_s}{\partial s}.\label{eq:diff_F_2}
	\end{align}

	Differentiating $\partial F_s/\partial Q (q_s(Q,P),Q) = P$ with respect to $Q$ and $P$ yields
	\begin{align}
		\frac{\partial^2 F_s}{\partial Q\, \partial q} \cdot \frac{\partial q_s}{\partial Q} + \frac{\partial^2 F_s}{\partial Q^2} &= 0,\label{eq:diff_F_3}\\
		\frac{\partial^2 F_s}{\partial q\, \partial Q} \cdot \frac{\partial q_s}{\partial P} &= 1.\label{eq:diff_F_4}
	\end{align}
	Therefore,
	\begin{align*}
		\frac{\partial H_s}{\partial Q} &= -\frac{\partial^2 F_s}{\partial q\, \partial s} \cdot \frac{\partial q_s}{\partial Q} - \frac{\partial^2 F_s}{\partial Q \, \partial s} \underset{\substack{\eqref{eq:diff_F_1}\\\eqref{eq:diff_F_2}}}{=} \left[ \frac{\partial^2 F_s}{\partial q\, \partial Q} \cdot \frac{\partial q_s}{\partial Q} + \frac{\partial^2 F}{\partial Q^2}\right]\cdot \frac{\partial Q_s}{\partial s} - \frac{\partial P_s}{\partial s} \underset{\eqref{eq:diff_F_3}}{=} -\frac{\partial P_s}{\partial s},\\
		\frac{\partial H_s}{\partial P}  &= -\frac{\partial^2F_s}{\partial q\, \partial s} \cdot \frac{\partial q_s}{\partial P} \underset{\eqref{eq:diff_F_1}}{=} \frac{\partial^2 F_s}{\partial q\, \partial Q} \cdot \frac{\partial q_s}{\partial P} \cdot \frac{\partial Q_s}{\partial s} \underset{\eqref{eq:diff_F_4}}{=} \frac{\partial Q_s}{\partial s},
	\end{align*}
	which exhibits $\psi_s$ as the path generated by $H_s$ (without any boundary condition).

	Since every billiard table $\gamma_s$ has positive curvature it follows that the family $\psi_s$ as well as the family $\psi_s^{-1}$ can be continuously extended to a map $\bar{A} \times [0,1] \to \bar{A}$ by setting $\psi_s|_{\partial A} = \id_{\partial A}$ and $\psi_s^{-1}|_{\partial A} = \id_{\partial A}$. Thus, we only need to prove that $H_s(Q,P)$ can be continuously extended to $\bar{A} \times [0,1]$ by setting $H_s|_{\partial A} = 0$.	To see that such an extension is indeed continuous, consider the inequality
	\[
		|H_s(Q,P)| = \left|\frac{\partial F_s}{\partial s}(q_s(Q,P),Q)\right| \leq \left\|\frac{\partial \gamma_s}{\partial s} (q_s(Q,P)) - \frac{\partial \gamma_s}{\partial s}(Q)\right\|.
	\]
	The right-hand side is continuous on $[0,1]\times \bar{A}$ and equals $0$ when $P = \pm 1$.
\end{proof}

\begin{proof}[Proof of Theorem~\ref{thm:main}]
	Since any two billiard tables can be joined by a path of billiard tables, the first statement of the theorem follows from Lemma~\ref{lem:ham_gen_fund}. Next, consider an arbitrary path $\gamma_s$ of billiard tables joining billiard tables $\alpha$ and $\beta$. Let $\psi_s$ be the corresponding Hamiltonian path of symplectomorphisms of the annulus $A$, and let $H_s \colon A \to \mathbb{R}$ be the smooth family of Hamiltonian functions corresponding to $\psi_s$. By Lemma~\ref{lem:ham_gen_fund} the following inequality holds
		\begin{multline*}
			l_{H}(\psi_s) = \int_0^1 \left(\newsup_{A}H_s - \inf_{A}H_s\right)\, ds = \\
			= \int_0^1 \left(\newsup_{S^1\times S^1 \setminus \Delta_{S^1}} \frac{\partial F_s}{\partial s}  - \inf_{S^1\times S^1 \setminus \Delta_{S^1}} \frac{\partial F_s}{\partial s}\right)\, ds \leq 2 \cdot \int_0^1 \newsup_{S^1\times S^1 \setminus \Delta_{S^1}} \left|\frac{\partial F_s}{\partial s}\right| \, ds,
		\end{multline*}
		where
		\[
			 F_s (q,Q) = \|\gamma_s(q) - \gamma_s(Q)\|.
		\]
		Using the triangle inequality, we obtain
		\[
			\newsup_{S^1\times S^1 \setminus \Delta_{S^1}} \left|\frac{\partial F_s}{\partial s}\right| \leq \newsup_{S^1\times S^1 \setminus \Delta_{S^1}} \left\|\frac{\partial\gamma_s}{\partial s}(q) - \frac{\partial\gamma_s}{\partial s}(Q)\right\|  \leq 2 \cdot \max_{S^1} \left\|\frac{\partial \gamma_s}{\partial s}\right\|.
		\]
		Therefore
		\[
			l_{H}(\psi_s) \leq 4 \cdot \int_0^1 \max_{S^1} \left\|\frac{\partial \gamma_s}{\partial s}\right\| \, ds \leq 4\cdot l_{B}(\gamma_s).
		\]
		Thus
		\[
			d_H(\psi_\alpha,\psi_\beta) \leq 4\cdot d_{B} (\alpha,\beta).
		\]
\end{proof}

	Due to Theorem~\ref{thm:main}, we have the Lipschitz map $\mathcal{B} \to \BM$. However, this map is not injective. To avoid this, it is more natural to consider $\mathcal{AB}$ instead of $\mathcal{B}$.

	\begin{lemma}\label{lem:inject}
	 The map $(\mathcal{AB}, d_{AB}) \to (\BM, d_{H})$ is injective and Lipschitz.
	\end{lemma}
	\begin{proof}
		Let $\gamma_{1}$ and $\gamma_{2}$ be two billiard tables and assume that the corresponding billiard ball maps $\psi_1,\psi_2 \colon A \to A$ coincide. Up to translations and rotations we can assume that $\gamma_1(0) = \gamma_2(0) = 0$ and $\gamma_1'(0) = \gamma_2' (0) = 1 \in \mathbb{C}$. Since $\psi_1$ and $\psi_2$ coincide we have
		\[
				\gamma_1'(Q(0,p)) = \gamma_2'(Q(0,p)) = e^{i(\varphi + \theta)} \in \mathbb{C},
		\]
		for every $p \in (-1,1)$ where the angles $\varphi,\theta \in (0,\pi)$ are determined by
		\[
			p = \cos \varphi,\quad P(0,p) = \cos \theta.
		\]
		Hence, $\gamma_1'(Q) = \gamma_2'(Q)$ for every $Q \in (0,1)$. Therefore, integration confirms $\gamma_1(Q) = \int_0^Q \gamma_1'(\tilde{Q}) d\tilde{Q}$ $= \int_0^Q \gamma_2'(\tilde{Q}) d\tilde{Q} = \gamma_2(Q)$. Thus, the map $(\mathcal{AB}, d_{AB}) \to (\BM, d_{H})$ is injective. The fact that this map is Lipschitz follows from Theorem~\ref{thm:main}.




		%

	\end{proof}

\medskip

We conclude this section with an application of Theorem 
\ref{thm:main}. 

\begin{corollary}\label{cor-dyn}
Let $X$ be a subset of the annulus $A$ such that the displacement energy $e(X)$ is strictly positive. Fix a positive integer $n$. Consider two billiard tables $\alpha, \beta$ and the corresponding billiard ball maps $\psi_{\alpha}$ and $\psi_{\beta}$. If $d_B(\alpha,\beta) < e(X)/4n$, then there is a point $x \in X$ with $\psi_\beta^n (x) \in \psi_\alpha^n(X)$.
\end{corollary}

For instance, $X$ could be an arbitrary nonempty open subset of $A$ (cf. Remark~\ref{rem:nondegeneracy} and Remark~\ref{rem:energy_compare}).

\begin{proof} Assume on the contrary that  $\psi_{\alpha}^n(X) \cap \psi_{\beta}^n(X) = \varnothing$, i.e.,
$X \cap \psi_\beta^{-n} \circ \psi_\alpha^n(X) = \varnothing$. 
Put $\psi_\alpha = \psi_\beta \circ \varphi$, with
$\phi \in \Ham(A,\partial A)$. We have
$$\psi_\alpha^n = \psi_\beta^n \prod_{j=n-1}^{0} \psi_\beta^{-j} \circ \phi \circ \psi_\beta^{j} \;.$$
The Hofer metric is invariant under conjugations by elements from $\Symp(A,\partial A)$.
Thus, by the triangle inequality and  Theorem \ref{thm:main}  we have 
 \[
 		d_{H}(\id_A, \psi_\beta^{-n} \circ \psi_\alpha^n)  \leq n\cdot d_{H}(\psi_{\alpha}, \psi_{\beta}) \leq 4n\cdot d_{B}(\alpha,\beta).
 \]
 Therefore, by definition of the displacement energy, 
 \[
 		0 < e(X) \leq 4n \cdot d_{B}(\alpha,\beta),
 \]
 which is a contradiction. 
 \end{proof}
  
This is a version of an argument  which appears in~\cite[proof of Lemma 2.1.10]{polterovich2014function}. We refer to Section \ref{sec:discussion} for further discussion of the corollary.

\section{Approximation of polygons}\label{sec:approximation_of_polygons}
	\begin{definition}
		Let $\mathcal{P}$ be the set of pairs $(P, S)$, where $P \subset \mathbb{R}^2$ is a convex polygon whose boundary $\partial P$ is of length $1$, and $S$ is a point on the boundary $\partial P$. The set $\mathcal{AP}$ is the quotient of $\mathcal{P}$ under the action of $SE(2)$.
	\end{definition}

	\begin{theorem} \label{thm:approx}
		There is a ``natural'' map $\mathcal{AP} \to \overline{\mathcal{AB}}$. More precisely, for every element of $\mathcal{AP}$ there is a $d_{AB}$-Cauchy sequence $\alpha_n$ of smooth billiard tables with positive curvature that tends in the $C^0$-norm to that element of $\mathcal{AP}$ (see the approximation construction below). In addition, all sequences obtained by the approximation procedure described below are $d_{AB}$-equivalent.
	\end{theorem}

	In the remaining part of this section we prove this theorem.

	\medskip

	For technical purposes we introduce the set of (marked) smooth convex bodies that need not be strictly convex.

	\begin{definition}
		Let $\mathcal{C}$ be the set of pairs $(K,S)$, where $K \subset \mathbb{R}^2$ is a smooth convex body (not necessarily strictly convex) whose boundary $\partial K$ is of length $1$, and $S$ is a point on the boundary $\partial K$.  We endow the set $\mathcal{C}$ with the metric $d_{C}$ defined in the same way as $d_{B}$, but instead of considering paths in $\mathcal{B}$ we consider paths in $\mathcal{C}$. Let $\mathcal{AC}$ be the quotient of $\mathcal{C}$ under the action of $SE(2)$ with the induced metric $d_{AC}$.
	\end{definition}

	Note that $\mathcal{AB} \subset \mathcal{AC}$. Moreover, the following two facts hold:
	\begin{itemize}
    	\item $\mathcal{AB}$ is dense in $(\mathcal{AC}, d_{\mathcal{AC}})$;
    	\item the restriction of $d_{\mathcal{AC}}$ to $\mathcal{AB}$ coincides with $d_{\mathcal{AB}}$.
	\end{itemize}
	Together, these facts imply that
	\[
  	  \overline{\mathcal{AB}} = \overline{\mathcal{AC}}.
	\]
	They follow from the following approximation properties:
	\begin{itemize}
    \item for every element $\alpha \colon S^1 \to \mathbb{R}^2$ from $\mathcal{C}$ there exists a smooth family $\Gamma_t \colon S^1 \to \mathbb{R}^2$ for $t \in [0,\varepsilon)$ such that $\Gamma_0 = \alpha$ and $\Gamma_t \in \mathcal{B}$ for $t>0$,
    \item for every smooth homotopy $\gamma_s \in \mathcal{C}$ between $\alpha, \beta \in \mathcal{B}$ there exists a smooth family $\Gamma_{t,s} \colon S^1 \to \mathbb{R}^2$ for $t \in [0,\varepsilon)$ and $s \in [0,1]$ such that $\Gamma_{0,s} = \gamma_s$, $\Gamma_{t,0} = \alpha$, $\Gamma_{t,1} = \beta$ for all $s$ and $t$; and $\Gamma_{t,s} \in \mathcal{B}$ for $t>0$.
\end{itemize}

	\begin{remark}\label{rem:approx_smooth}
		Let us give brief justifications for these properties here:
	\begin{itemize}
	 \item Let $\alpha \colon S^1 \to \mathbb{R}^2$ be a curve from $\mathcal{C}$. Without loss of generality assume that the origin is in the interior of the convex body bounded by $\alpha$ and consider the family of curves $\widetilde{\Gamma}_{t} \colon S^1 \to \mathbb{R}^2$ for $t \geq 0$ such that
	 \[
	 		\widetilde{\Gamma}_{t}(q) = \frac{\alpha(q)}{\sqrt{1+ t\|\alpha(q)\|^2}}.
	 \]
	 Note that if $\mu_{\alpha} \colon \mathbb{R}^2 \to \mathbb{R}$ is the Minkowski functional for the convex body bounded by $\alpha$, then $\widetilde{\Gamma}_{t}(q)$ parametrizes the boundary of a convex body with Minkowski functional $\sqrt{\mu_\alpha(x)^2 + t \|x\|^2}$. Direct calculation shows that $\widetilde{\Gamma}_{t}$ has positive curvature for $t>0$. This family is smooth and $\Gamma_0 = \alpha$. The curves $\widetilde{\Gamma}_{t}$ need not have length one for $t> 0$. We therefore normalize them by setting
	\[
		\widehat{\Gamma}_t = \frac{\widetilde{\Gamma}_{t}}{\Length(\widetilde{\Gamma}_{t})}.
	\]
	Now every curve $\widehat{\Gamma}_t$ has length one and $\widehat{\Gamma}_0 = \alpha$. We then reparametrize each normalized curve by arc length, starting from $\widehat{\Gamma}_{t}(0)$ to obtain the desired smooth family $\Gamma_t$. Note that $\Gamma_t \in \mathcal{B}$ for $t>0$ and $\Gamma_0 = \alpha$.

	\item  Let $\gamma_s \in \mathcal{C}$ be a smooth homotopy between $\alpha, \beta \in \mathcal{B}$. Let $O_s = \int_{S^1} \gamma_s(q)\,dq$. It is easy to check that $O_s$ is an interior point of the convex body bounded by $\gamma_s$. Let $\chi \colon [0,1] \to [0,1]$ be a smooth function such that $\chi(0) = \chi(1) = 0$ and $\chi(s) > 0$ for $s \in (0,1)$. Consider the smooth family of curves $\widetilde{\Gamma}_{t,s} \colon S^1 \to \mathbb{R}^2$ for $t \geq 0$ and $s \in [0,1]$

	\[
		\widetilde{\Gamma}_{t,s}(q) = O_s + \frac{\gamma_s(q) - O_s}{\sqrt{1 + t \chi(s)\cdot \|\gamma_s(q)-O_s\|^2}}.
	\]

	Note that $\widetilde{\Gamma}_{t,0} = \alpha$, $\widetilde{\Gamma}_{t,1} = \beta$ and $\widetilde{\Gamma}_{0,s} = \gamma_s$ and every curve $\widetilde{\Gamma}_{t,s}$ has positive curvature for $t > 0$.
	Normalizing and reparametrizing as in the previous item yields the desired smooth family $\Gamma_{t,s}$.
	\end{itemize}
	\end{remark}

	Now we are ready to describe the process of embedding $\mathcal{AP} \to \overline{\mathcal{AB}}$. Let $\eta \in \mathcal{P}$ and $K$ be the corresponding polygon. We assume that $S = \eta(0)$ is not a corner of the polygon (the case when $\eta(0)$ is a corner will be explained later). Enumerate the corners  of $K$ counterclockwise  and denote them by $D_1, \ldots, D_n$. Assume that $S \in [D_n,D_1]$. For each $D_i$ consider the orthonormal coordinate system with the origin at $D_i$ such that $\partial K$ locally looks like a graph of the function $a_i |x|$ for some $a_i > 0$. Take some  smooth convex function $f_i(x)$ such that $f_i(x)$ coincides with $a_i |x|$ outside a small neighborhood of $0$. Using this family of functions $(f_1, \ldots ,f_n)$ we can construct the corresponding smooth family of positively oriented convex Jordan curves $\gamma_s \colon S^1 = \mathbb{R}/\mathbb{Z} \to \mathbb{R}^2$ for $s \in (0,1]$ such that
	\begin{itemize}
		\item $\gamma_{s}$ coincides with the $\graph(s \cdot f_i(x/s))$ in a small neighborhood of $D_i$ (in the corresponding coordinate system)
		\item Outside all these neighborhoods $\gamma_{s}$ coincides with $\partial K$
		\item $\gamma_s(0) = S$
		\item $\left\|\partial\gamma_{s}/\partial q\right\| = \length(\gamma_{s})$.
	\end{itemize}

	\begin{figure}[ht]
		\begin{subfigure}{.42\textwidth}
  		\centering
  		\includegraphics[width=0.97\linewidth]{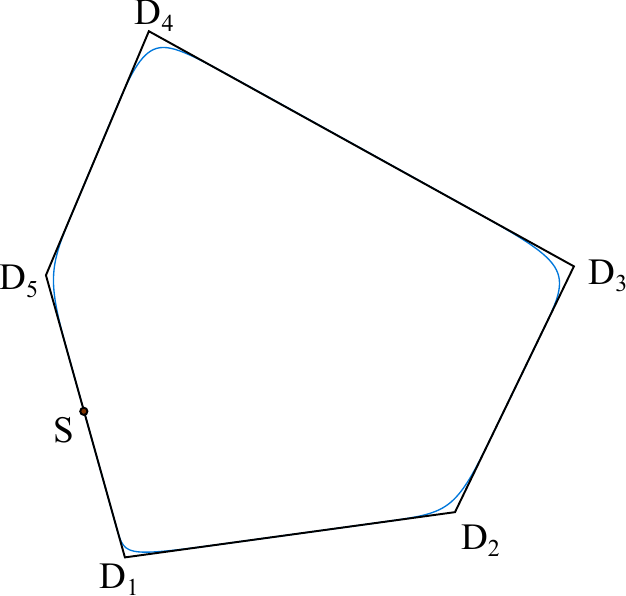}
		\end{subfigure}%
		\begin{subfigure}{.56\textwidth}
  		\centering
  		\includegraphics[width=0.97\linewidth]{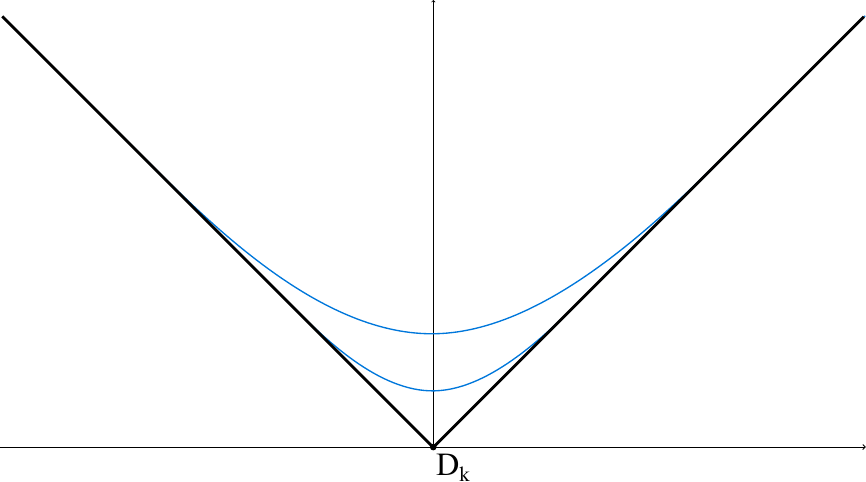}
		\end{subfigure}
		\caption{Approximation of a polygon.}
		\label{fig:approximation}
	\end{figure}

	Denote $L(s) = \length(\gamma_{s})$ and $\lambda(s) = 1/L(s)$. Take an arbitrary point $O$ in $\mathbb{R}^2$ and consider the family of curves $\widetilde{\gamma}_s = \lambda(s)\cdot\gamma_s$, where $O$ is the center of homothety; $\widetilde{\gamma}_{s}$ is a smooth path in $\mathcal{C}$ for $s \in (0,1]$.

	\begin{theorem} \label{thm:approx_es_ind}
		There exists a limit of $[\widetilde{\gamma}_s]$ in $\overline{\mathcal{AC}} = \overline{\mathcal{AB}}$ when $s \to 0$ and this limit is independent of the choice of the functions $f_1, \ldots, f_n$ and the center of homothety $O$.
	\end{theorem}
	\begin{proof}
		\emph{(a) Existence of the limit}. To prove that there exists a limit of $[\widetilde{\gamma}_s]$ in $\overline{\mathcal{AC}}$ when $s \to 0$ we will show that
		\begin{equation}\label{main_int:fin}
			\int_0^1 \max_{q \in S^1} \left\| \frac{\partial \widetilde{\gamma}_s}{\partial s}(q) \right\| ds < +\infty.
		\end{equation}
		To prove this, it is enough to demonstrate that
		\begin{equation}\label{second_int:fin}
			\int_0^1 \max_{q \in S^1} \left\| \frac{\partial \gamma_s}{\partial s}(q) \right\| ds < +\infty.
		\end{equation}
		Indeed, note that
		\begin{equation}\label{Leibniz:fin}
			 \frac{\partial \widetilde{\gamma}_s}{\partial s} = \frac{\partial \lambda}{\partial s} (s)\cdot \gamma_s + \lambda(s) \cdot\frac{\partial \gamma_s}{\partial s}.
		\end{equation}
		Since $\lambda(s)$ is increasing, we have
		\begin{equation}\label{Lambda:monotone}
			\int_0^1 \left|\frac{\partial \lambda}{\partial s}(s)\right| ds = \lambda(1) - \lim_{s \to 0} \lambda(s) \leq \lambda(1).
		\end{equation}
		
		Now~\eqref{main_int:fin} can be inferred from equations~\eqref{second_int:fin},~\eqref{Leibniz:fin},~\eqref{Lambda:monotone} and the fact that $|\lambda(s)|$ and $\|\gamma_s(q)\|$ are bounded. Therefore, we only have to show~\eqref{second_int:fin}.

		\begin{figure}[ht]
  		\centering
  		\includegraphics[width=0.5\textwidth]{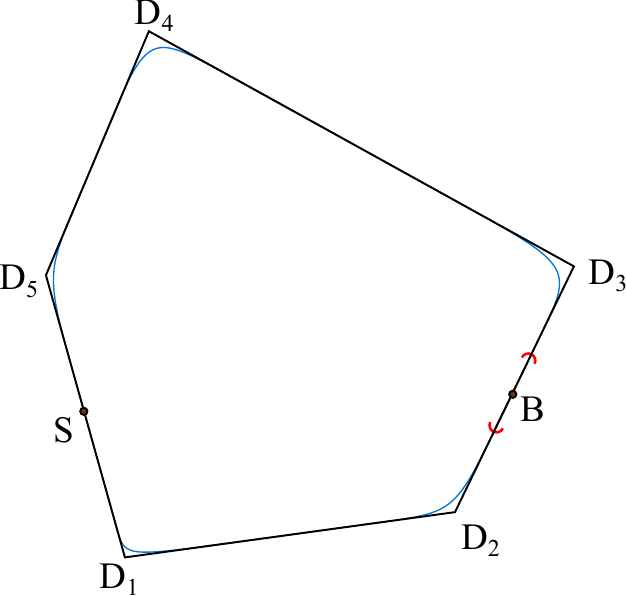}
			\caption{Local picture on an edge.}
			\label{fig:ponit_on_an_edge}
		\end{figure}

		Let $B$ be an interior point of some edge of $\partial K$. Let $(q_0,s_0)$ be a point such that $\gamma_{s_0}(q_0) = B$ and $\gamma_{s_0}$ coincides with $\partial K$ in a small neighborhood of $B$. We want to compute $\partial \gamma_s/\partial s$ at the point $(q_0,s_0)$. Denote by $\beta(q)$ the arc-length counterclockwise parametrization of $\partial K$ in some neighborhood of the point $B$ satisfying $\beta(0) = B$. In this case for $(q,s)$ close enough to $(q_0,s_0)$ one has
		\begin{equation}\label{gamma_on_an_edge}
		\gamma_s(q) = \beta((q-Q_B(s)) \cdot L(s)),
		\end{equation}
		 where $Q_B (s)$ is uniquely defined by the equation $\gamma_s(Q_B(s)) = B$. Indeed, for a fixed $s$ both curves have the same initial conditions $\gamma_s(Q_B(s)) = \beta (0) = B$ and the same speed $L(s)$.

		 We can describe $Q_B(s)$ in a slightly different way. Let $l_{SB}(s)$ be the distance between $S$ and $B$ along $\gamma_s$ in the counterclockwise direction, then
		 \begin{equation}\label{eq:Q_B}
		 Q_B(s) = \frac{l_{SB}(s)}{L(s)}
		 \end{equation}
		 For our computations we need the explicit forms of $l_{SB}(s)$ and $L(s)$. Let $\delta_i$ be the difference between the lengths of $\partial K$ and $\gamma_1$ near the corner $D_i$. Then
		\begin{equation}\label{eq:l_SB}
			l_{SB}(s) = l_{SB,\partial K} - s \cdot \sum_{i=1}^k \delta_i,
		\end{equation}
		and
		\begin{equation}\label{eq:Ls}
			L(s) = L_{\partial K} - s \cdot \sum_{i=1}^n \delta_i,
		\end{equation}
		where $L_{\partial K}$ is the length of $\partial K$ and $l_{SB,\partial K}$ is the distance between $S$ and $B$ along $\partial K$ in the counterclockwise direction and $k$ is the number of corners between $S$ and $B$ in the counterclockwise direction. From~\eqref{gamma_on_an_edge}
		\[
			\frac{\partial \gamma_s}{\partial s}(q) = \frac{\partial \beta}{\partial q} (\ldots) \cdot \left[-\frac{\partial Q_B}{\partial s}(s) \cdot L(s) + \frac{\partial L}{\partial s} (s) \cdot (q- Q_B(s))\right].
		\]
		Note that $\beta$ is arc-length parametrized and $Q_B(s_0) = q_0$, therefore,
		\begin{equation}\label{eq:partgamma}
			\left\| \frac{\partial \gamma_s}{\partial s} \right\|_{(q_0,s_0)} = \left|L(s_0) \cdot \frac{\partial Q_B}{\partial s}(s_0)  \right| =  \frac{\left|\frac{\partial l_{SB}}{\partial s}(s_0)\cdot L(s_0) - \frac{\partial L}{\partial s}(s_0)\cdot l_{SB}(s_0)\right|}{L(s_0)},
		\end{equation}
		where the last equality follows from~\eqref{eq:Q_B}.

		Combining~\eqref{eq:l_SB}, \eqref{eq:Ls}, and~\eqref{eq:partgamma}, one can obtain
		\[
			 \left\| \frac{\partial \gamma_s}{\partial s} \right\|_{(q_0,s_0)} \leq \frac{\bigg|L(s_0) \cdot \sum_{i=1}^k \delta_i\bigg|+\bigg|l_{SB}(s_0)\cdot \sum_{i=1}^n \delta_i \bigg|}{L(s_0)} \leq \frac{2L_{\partial K} \cdot \sum_{i=1}^n \delta_i}{L_{\partial K} - \sum_{i=1}^n \delta_i}.
		\]
		Note that the right-hand side is independent of the choice of $(q_0,s_0)$ provided $\gamma_{s_0}(q) \in \partial K$ in a small neighborhood of $q_0$.

		\begin{figure}[ht]
  		\centering
  		\includegraphics[width=0.5\textwidth]{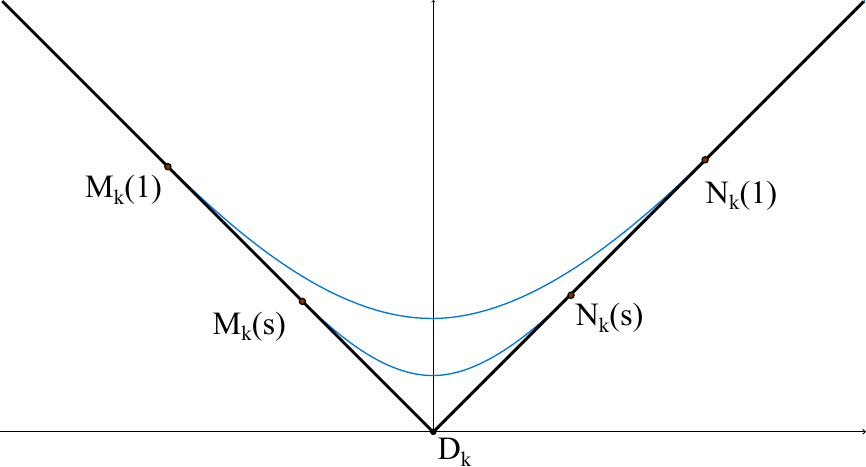}
			\caption{Local picture near a corner.}
			\label{fig:corner_computations}
		\end{figure}

		Our next goal is to analyze the behavior of $\partial \gamma_s/\partial s$ near any corner. Fix a corner $D_k$ and for every $s\in (0,1]$ consider the pair of points $M_{k} (s)$ and $N_{k} (s)$, where $M_{k}(s)$ is the closest point to $D_k$ on $[D_{k-1},D_k] \cap \gamma_{s}$ and $N_{k}(s)$ is the closest point to $D_k$ on $[D_{k},D_{k+1}] \cap \gamma_{s}$ (see Figure~\ref{fig:corner_computations}). Moreover, let $h_k$ be the length of $\gamma_1$ between $M_{k}(1)$ and $N_{k}(1)$. Let $\alpha (q)$ be the arc-length reparametrization of $\gamma_1$ such that $\alpha(0) = M_{k}(1)$, and consider a coordinate system where $D_k$ is the origin. Then locally we have
		\[
			\gamma_s (q) = s \cdot \alpha\left((q-Q_k(s))\cdot\frac{L(s)}{s} \right),
		\]
		where $Q_k(s)$ is uniquely defined by the equation $\gamma_s (Q_k (s)) = M_{k}(s)$. Therefore,
		\[
			Q_k (s) =  \frac{l_{SD_k,\partial K} - s\cdot \left(\sum_{i=1}^{k-1} \delta_i + \|M_{k}(1) - D_k\|\right)}{L_{\partial K} - s \cdot \sum_{i=1}^n \delta_i}.
		\]
		Our goal is to estimate $\left\|\frac{\partial \gamma_s}{\partial s} (q)\right\|$ for $q \in \left[Q_k(s), Q_k(s) + \frac{h_k \cdot s}{L(s)} \right]$. First of all, for such $q$ we have
		\[
			\frac{\partial \gamma_s}{\partial s} (q) = \alpha(\ldots) + \frac{\partial \alpha}{\partial q} (\ldots) \cdot \left[s \cdot \frac{\partial}{\partial s} \left((q-Q_k(s))\cdot\frac{L(s)}{s} \right)\right].
		\]
		Then
		\[
			\left\|\frac{\partial \gamma_s}{\partial s} (q)\right\| \leq \max(\|M_{k}(1) - D_k\|, \|N_{k}(1) - D_k\|) + \left|s \cdot \frac{\partial}{\partial s} \left((q-Q_k(s))\cdot\frac{L(s)}{s} \right) \right|.
		\]
		Note that
		\[
			\max_{q \in \left[Q_k(s), Q_k(s) + \frac{h_k \cdot s}{L(s)} \right]}\left|s \cdot \frac{\partial}{\partial s} \left((q-Q_k(s))\cdot\frac{L(s)}{s} \right) \right| \leq L(s)\cdot \left|\frac{\partial Q_k}{\partial s} (s)\right| + \frac{h_k \cdot s^2}{L(s)} \cdot  \left|\frac{\partial}{\partial s} \frac{L(s)}{s} \right|.
		\]
		For the first summand we have
		\[
			L(s)\cdot \left|\frac{\partial Q_k}{\partial s} (s)\right| \leq \frac{2L_{\partial K} \cdot \left(\sum_{i=1}^n \delta_i + L_{\partial K}\right)}{L_{\partial K} - \sum_{i=1}^n \delta_i}.
		\]
		For the second summand we have
		\[
			\frac{h_k \cdot s^2}{L(s)} \cdot  \left|\frac{\partial}{\partial s} \frac{L(s)}{s} \right| \leq \frac{L_{\partial K}^2}{L_{\partial K} - \sum_{i=1}^n \delta_i}.
		\]
		Therefore, from the above estimates it follows that
		\[
			\max_{q \in S^1} \left\|\frac{\partial \gamma_s}{\partial s} (q)\right\| \leq C,
		\]
		where $C$ is independent of $s \in (0,1]$. This implies~\eqref{second_int:fin}.
		
		\emph{(b) Choice independence}. First of all, note that the approximation procedure is independent of the homothety center $O$, because if we scale a curve by the same factor with respect to different points, then the resulting curves differ by a shift and belong to the same class in $\mathcal{AC}$.
		
		Assume that $\gamma_{1,s}$ is generated by the family $(f_1,\ldots,f_n)$ and $\gamma_{0,s}$ by the family $(g_1,\ldots,g_n)$. Replacing each $g_i(x)$ by $\varepsilon g_i(x/\varepsilon)$, for a sufficiently small common $\varepsilon>0$, we may assume that $g_i\le f_i$ for every $i$. Indeed, this replacement merely reparametrizes the corresponding approximation by $s\mapsto\varepsilon s$, so its limit, whose existence was established in the previous part, remains unchanged. Our goal is to prove that the corresponding normalized families $\widetilde{\gamma}_{1,s}$ and $\widetilde{\gamma}_{0,s}$ satisfy
		\begin{equation}\label{eq:independ}
			d_{C}(\widetilde{\gamma}_{1,s},\widetilde{\gamma}_{0,s}) = O(s),
		\end{equation}
		as $s \to 0$.
		
		First of all, take smooth families $(\psi_{1,t},\ldots, \psi_{n,t})$ of convex functions for $t \in [0,1]$ such that $\psi_{i,0} = g_i$, $\psi_{i,1} = f_i$ and $\psi_{i,t}$ is monotone in $t$. For example, $\psi_{i,t} = (1-t)g_i + tf_i$. Then for each $t \in [0,1]$ we can construct the corresponding smooth family $\gamma_{t,s} \colon S^1 \to \mathbb{R}^2$ for $s \in (0,1]$. Let $\delta_k(t)$ be the difference between the lengths of $\partial K$ and $\gamma_{t,1}$ near the corner $D_k$. Note that every $\delta_k(t)$ is smooth on $[0,1]$ and monotonically increasing in $t$. Let $L(t,s) = \length(\gamma_{t,s})$, then
		\[
			L(t,s) = L_{\partial K} - s\cdot\sum_{i=1}^n\delta_{i}(t).
		\]
		In order to show~\eqref{eq:independ}, it is enough to prove
		\begin{equation}\label{second_eq:independ}
			\int_{0}^1 \max_{q \in S^1} \left\|\frac{\partial \gamma_{t,s}}{\partial t}(q) \right\| dt = O(s).
		\end{equation}
	 Indeed, introduce $\lambda(t,s) = 1/L(t,s)$ and $\widetilde{\gamma}_{t,s} = \lambda(t,s) \cdot \gamma_{t,s}$, then
		\[
			d_{C}(\widetilde{\gamma}_{1,s},\widetilde{\gamma}_{0,s}) \leq \int_{0}^1 \max_{q \in S^1} \left\|\frac{\partial \widetilde{\gamma}_{t,s}}{\partial t}(q) \right\| dt \leq O(s) + O(1)\cdot |\lambda(1,s) - \lambda(0,s)| = O(s).
		\]
		
		Let $B$ be an interior point of some edge of $\partial K$. Let $(q_0,t_0,s_0)$ be a point such that $\gamma_{t_0,s_0}(q_0) = B$ and $\gamma_{t_0,s_0}$ coincides with $\partial K$ in a small neighborhood of $B$. Denote by $\beta(q)$ the arc-length counterclockwise parametrization of $\partial K$ in some neighborhood of the point $B$ satisfying $\beta(0) = B$. Then for $(q,t,s)$ close enough to $(q_0,t_0,s_0)$ we have
		\begin{equation}\label{eq:gamma_ts}
			\gamma_{t,s}(q) = \beta((q-Q_B(t,s)) \cdot L(t,s)),
		\end{equation}
		where $Q_B (t,s)$ is uniquely defined by the equation $\gamma_{t,s}(Q_B(t,s)) = B$. Note that $Q_B(t,s)$ can also be expressed in the form
		\begin{equation}\label{eq:Q_Bts}
			Q_B(t,s) = \frac{l_{SB,\partial K} - s \cdot \sum_{i=1}^k\delta_{i}(t)}{L_{\partial K} - s\cdot\sum_{i=1}^n\delta_{i}(t)},
		\end{equation}
		where $k$ is the number of corners between $S$ and $B$ in counterclockwise direction. From~\eqref{eq:gamma_ts}
		\[
			\frac{\partial \gamma_{t,s}}{\partial t} = \frac{\partial \beta}{\partial q}(\ldots) \cdot \left[ -\frac{\partial Q_B}{\partial t} (t,s) \cdot L(t,s) + \frac{\partial L}{\partial t}(t,s) \cdot (q- Q_B(t,s)) \right].
		\]
		Note that $\beta$ is arc-length parametrized and $Q_B(t_0,s_0) = q_0$, therefore,

		\begin{equation}\label{eq:partgammateq}
			\left\|\frac{\partial \gamma_{t,s}}{\partial t}\right\|_{(q_0,t_0,s_0)} = L(t_0,s_0) \cdot \left|\frac{\partial Q_B}{\partial t} (t_0,s_0)\right|.
		\end{equation}
		Let us rewrite formula~\eqref{eq:Q_Bts} in the form
		\[
				Q_B(t,s) = \frac{l_{SB,\partial K} - s \cdot f(t)}{L_{\partial K} - s\cdot g(t)},
		\]
		where $f(t) = \sum_{i=1}^k\delta_{i}(t)$ and $g(t) = \sum_{i=1}^n\delta_{i}(t)$. Then
		\[
				\frac{\partial Q_B}{\partial t}(t,s) = \frac{s \cdot g'(t)\cdot [l_{SB,\partial K} - s \cdot f(t)] - s\cdot f'(t)\cdot [L_{\partial K} - s\cdot g(t)]}{[L_{\partial K} - s\cdot g(t)]^2}.
		\]
		Therefore,
		\begin{equation}\label{eq:Q_Bts_ineq}
			\left|\frac{\partial Q_B}{\partial t}(t_0,s_0)\right| \leq \frac{|g'(t_0)\cdot L_{\partial K}| + |f'(t_0)\cdot L_{\partial K}|}{[L(t_0,s_0)]^2} \cdot s_0 \leq \frac{2 \left[\sum\limits_{i=1}^n \delta_{i}'(t_0) \right] \cdot L_{\partial K}}{[L(t_0,s_0)]^2} \cdot s_0.
		\end{equation}
		In the last inequality we used the fact that every function $\delta_i (t)$ is monotonically increasing. Thus, by formula~\eqref{eq:partgammateq} and formula~\eqref{eq:Q_Bts_ineq} one can obtain
		\[
			 \left\|\frac{\partial \gamma_{t,s}}{\partial t}\right\|_{(q_0,t_0,s_0)} \leq  \frac{2 \left[\sum\limits_{i=1}^n \delta_{i}'(t_0) \right] \cdot L_{\partial K}}{L(t_0,s_0)} \cdot s_0 \leq \frac{2 \left[\sum\limits_{i=1}^n \delta_{i}'(t_0) \right] \cdot L_{\partial K}}{L_{\partial K} - \sum\limits_{i=1}^n \delta_{i}(1)} \cdot s_0 \leq C \cdot s_0,
		\]
		for some constant $C$ independent of the choice of such a $(q_0,t_0,s_0)$ since every function $\delta_i$ is smooth on $[0,1]$.

		Fix a corner $D_k$ and for every $s \in (0,1]$ consider the pair of points $M_{k}(s)$ and $N_{k}(s)$ where $M_{k}(s)$ is the closest point to $D_k$ on $[D_{k-1},D_k] \cap \gamma_{1,s}$ and $N_{k}(s)$ is the closest point to $D_k$ on $[D_{k},D_{k+1}] \cap \gamma_{1,s}$. Moreover, let $h_k(t) \in C^\infty([0,1])$ be the length of $\gamma_{t,1}$ between $M_{k}(1)$ and $N_{k}(1)$. Let $\alpha_t (q)$ be the arc-length reparametrization of $\gamma_{t,1}$ such that $\alpha_t(0) = M_{k}(1)$. Consider the coordinate system where $D_k$ is the origin. Then, locally, we have
		\[
			\gamma_{t,s} (q) = s \cdot \alpha_t \left((q-Q_k(t,s))\cdot\frac{L(t,s)}{s} \right),
		\]
		where
		\[
			Q_k (t,s) = \frac{l_{SD_k,\partial K} - s\cdot \left(\sum_{i=1}^{k-1} \delta_i(t) + \|M_{k}(1) - D_k\|\right)}{L_{\partial K} - s \cdot \sum_{i=1}^n \delta_i(t)}.
		\]
		Our goal is to estimate $\left\|\frac{\partial \gamma_{t,s}}{\partial t} (q)\right\|$ for $q \in \left[Q_k(t,s), Q_k(t,s) + \frac{h_k(t) \cdot s}{L(t,s)} \right]$. For such a $q$, we have
		\[
			\frac{\partial \gamma_{t,s}}{\partial t} (q) = s \cdot \frac{\partial\alpha_t}{\partial t}(\ldots) + \frac{\partial \alpha_t}{\partial q} (\ldots) \cdot \frac{\partial}{\partial t} \left[(q-Q_k(t,s))\cdot L(t,s)\right].
		\]
		For the first term, we have
		\[
			s \cdot \left\|\frac{\partial\alpha_t}{\partial t}(\ldots)\right\| \leq s \cdot \max_{q \in [0,h_k(t)]} \left\|\frac{\partial\alpha_t}{\partial t}(q)\right\| \leq C\cdot s.
		\]
		For the second term, we have
		\begin{align*}
			&\left\|\frac{\partial \alpha_t}{\partial q} (\ldots) \cdot \frac{\partial}{\partial t} \left[(q-Q_k(t,s))\cdot L(t,s)\right]\right\| = \left|\frac{\partial}{\partial t} \left[(q-Q_k(t,s))\cdot L(t,s)\right]\right| \leq \\
			&\leq \frac{h_k(t) \cdot s}{L(t,s)} \cdot\left|\frac{\partial L}{\partial t}(t,s)\right| + L(t,s) \cdot \left|\frac{\partial Q_k}{\partial t}(t,s)\right| \leq C \cdot s.
		\end{align*}
		Thus,
		\[
			\max_{\substack{q \in S^1\\ t \in [0,1]}} \left\|\frac{\partial \gamma_{t,s}}{\partial t}(q)\right\| = O(s).
		\]
		This completes the proof of Theorem~\ref{thm:approx_es_ind}.
	\end{proof}

	Therefore, we have defined an element $\Psi([\eta(q)]) \in \overline{\mathcal{AC}}$ for every $[\eta(q)] \in \mathcal{AP}$ such that $\eta(0)$ is not a corner. The construction is compatible with shifts of the marked point whenever both markings lie away from the corners. Moreover, for every $\gamma \in \mathcal{C}$ the following inequality holds $d_{AC}([\gamma(q+a)], [\gamma(q+b)]) \leq |a-b|$ for every $a,b \in \mathbb{R}$, since shifting the parameter from \(a\) to \(b\) defines a path of length at most $|a-b|$. This estimate passes to the completion. Consequently, for every fixed polygon \(\eta\), the map $r \mapsto \Psi([\eta(q+r)])$ is $1$-Lipschitz on the set of parameters corresponding to noncorner markings. In particular, if $\eta(0)$ is a corner the limit $\lim_{r \to 0} \Psi([\eta(q + r)]) \in \overline{\mathcal{AC}}$ is well-defined and we define $\Psi([\eta(q)])$ to be equal to this limit. Thus, we have an $S^1$-equivariant map $\Psi \colon \mathcal{AP} \to \overline{\mathcal{AC}} = \overline{\mathcal{AB}}$, which completes the proof of Theorem~\ref{thm:approx}. 

	\begin{theorem}
		The map $\mathcal{AP} \to \overline{\mathcal{AB}}$ constructed above is injective.
	\end{theorem}

	\begin{proof}
		Let $\beta_1 \colon S^1 \to \mathbb{R}^2$ and $\beta_2 \colon S^1 \to \mathbb{R}^2$ be two elements in $\mathcal{P}$ representing distinct elements of $\mathcal{AP}$, meaning that there is no $T \in \SE(2)$ such that $\beta_1 = T\beta_2$. Consider the sequences $\gamma_{1,n} \in \mathcal{B}$ and $\gamma_{2,n} \in \mathcal{B}$ such that $d_{C^0}(\gamma_{i,n}, \beta_i) \to 0$ as $n \to \infty$ for $i=1,2$ and such that $[\gamma_{i,n}]$ represents the image of $[\beta_i]$ under the map $\mathcal{AP} \to \overline{\mathcal{AB}}$. Assume that $[\gamma_{1,n}]$ and $[\gamma_{2,n}]$ represent the same element of $\overline{\mathcal{AB}}$. This means that $d_{AB} ([\gamma_{1,n}], [\gamma_{2,n}]) \to 0$. By the definition of $d_{AB}$ there exists a sequence $T_n \in \SE(2)$ such that $d_{B}(\gamma_{1,n}, T_n(\gamma_{2,n})) \to 0$. By Lemma~\ref{lem:compare_bil_c0} it follows that $d_{C^0}(\gamma_{1,n}, T_n(\gamma_{2,n})) \to 0$. Therefore, from the triangle inequality we have
		\begin{align*}
			d_{C^0} (\beta_1, T_n \beta_2) &\leq d_{C^0} (\beta_1, \gamma_{1,n}) + d_{C^0}(\gamma_{1,n}, T_n\gamma_{2,n}) + d_{C^0}(T_n\gamma_{2,n}, T_n\beta_2) =\\
			&= d_{C^0} (\beta_1, \gamma_{1,n}) + d_{C^0}(\gamma_{1,n}, T_n\gamma_{2,n}) + d_{C^0}(\gamma_{2,n}, \beta_2) \to 0
		\end{align*}
		Hence, $d_{C^0} (\beta_1, T_n \beta_2) \to 0$. Under the decomposition $\SE(2) \cong \SO(2) \times \mathbb{R}^2 \cong S^1 \times \mathbb{R}^2$, the translation components of $T_n$ are bounded. Hence, after passing to a subsequence, we may assume that $T_n \to T \in \SE(2)$. It follows that $d_{C^0} (\beta_1, T \beta_2) = 0$, and thus $\beta_1 = T \beta_2$, contradicting our assumption.
	\end{proof}

	\section{Comparison of the metric \texorpdfstring{$d_{B}$}{TEXT} with other geometric distances}	\label{sec:comparison}

	In this section, we discuss two estimates for the geometric distance $d_B$. First, note that $d_B$ is greater than or equal to the standard $C^0$ distance.

	\begin{lemma}\label{lem:compare_bil_c0}
		Let $\alpha$ and $\beta$ be two billiard tables. Then
		\[
  		d_B(\alpha,\beta) \geq \max_{q \in S^1} \|\alpha(q) - \beta(q) \|
		\]
	\end{lemma}
	\begin{proof}
		Consider a path $\gamma_s$ of billiard tables joining $\alpha$ and $\beta$. Then
		\[
 			\int_{0}^1 \max_{q \in S^1} \left\|\frac{\partial \gamma_s}{\partial s} (q)\right\| ds \geq \max_{q \in S^1} \int_{0}^1 \left\|\frac{\partial \gamma_s}{\partial s} (q)\right\|ds \geq \max_{q \in S^1} \|\alpha(q) -\beta(q)\|.
		\]
		This inequality holds for every path $\gamma_s$. Therefore, taking the infimum over all such paths, we obtain the desired inequality.
\end{proof}

	\begin{remark}
		Let $\alpha,\beta \colon S^1 \to \mathbb{R}^2$ be billiard tables, and let $K(\alpha)$ and $K(\beta)$ denote the corresponding convex bodies. Let $d_{\text{Hausd}} (K(\alpha), K(\beta))$ be the Hausdorff distance between $K(\alpha)$ and $K(\beta)$. The following inequality
		\[
			d_{\text{Hausd}} (K(\alpha), K(\beta)) \leq \max_{q \in S^1} \|\alpha(q) - \beta(q) \|
		\]
		follows from the definition of the Hausdorff distance. Hence, 
		\[
		d_{\text{Hausd}} (K(\alpha), K(\beta)) \leq d_{B}(\alpha,\beta).
		\]
		Therefore, every function on billiard tables that is continuous with respect to the Hausdorff metric is continuous with respect to $d_{B}$ as well. For example, the sequence of $p$-widths of a billiard table, an interesting family of invariants defined in~\cite{chodosh2025p}, is continuous in both metrics\footnote{Note that the $p$-width is $1$-homogeneous, invariant under translations, and monotone with respect to inclusion~\cite[Lemma 2.4]{chodosh2025p}. These properties imply the continuity of the $p$-width with respect to the Hausdorff metric. The proof is similar to the argument in~\cite[proof of Theorem 1.8.20]{Schneider_2013}.}.
	\end{remark}


	Now, let $\alpha \colon S^1 \to \mathbb{R}^2$ be a billiard table and $n(q)$ be the outward unit normal vector field on $\alpha$. Let $f \colon S^1 \to \mathbb{R}$ be a smooth function. If $f$ is $C^2$-small enough, then every curve of the form
	\[
		\gamma_s(q) = \alpha (q) + s\cdot f(q)\cdot n(q)
	\]
	has positive curvature for every $s \in [0,1]$. In addition, assume that $\gamma_1$ has length one. Then after arc-length reparametrization, with the fixed starting point $\gamma_1(0)$, we obtain a billiard table~$\beta$. Our goal is to prove an inequality of the form
	\[
		d_{B} (\alpha,\beta) \leq C \cdot \left( \max_{q \in S^1} |f(q)| + \max_{q \in S^1}\left|\frac{\partial f}{\partial q} (q)\right| \right),
	\]
	where $C$ depends only on $\alpha$. Let $L(s)$ be the length of $\gamma_s$. Then	
	\[
  	L(s) = \int_0^1 \left\|\frac{\partial \gamma_s}{\partial q}\right\| dq = \int_0^1\left\|\frac{\partial \alpha}{\partial q} + s \frac{\partial f}{\partial q} n + sf\frac{\partial n}{\partial q}\right\|dq =\int_0^1 \sqrt{(1 + s f\cdot k)^2 + \left(s\cdot\partial f/\partial q\right)^2}dq,
	\]
	where $k$ is the curvature of $\alpha$. Denote $\alpha(0)$ by $O$ and consider the following family
	\[
  	\widetilde{\gamma}_s(q) = O + \frac{\gamma_s(q)-O}{L(s)},
	\]
	then every curve in this family has length one. Consider an arc-length reparametrization $\hat{\gamma}_s$ of every curve $\widetilde{\gamma}_s$ such that $\hat{\gamma}_s(0) = \widetilde{\gamma}_s(0)$, for every $s$. Let us introduce the function
	\[
  	l(s,q) = \int_0^q \left\|\frac{\partial \widetilde{\gamma}_s}{\partial q}\right\|dq = \frac{\int_0^q \sqrt{(1 + s f\cdot k)^2 + \left(s\cdot\partial f/\partial q\right)^2}dq}{\int_0^1 \sqrt{(1 + s f\cdot k)^2 + \left(s\cdot\partial f/\partial q\right)^2}dq},
	\]
	then $\hat{\gamma}_s(l(s,q)) = \widetilde{\gamma}_s(q)$ and one can show that
	\[
		\max_{q \in S^1} \left\|\frac{\partial \hat{\gamma}_s}{\partial s}(q)\right\| \leq \max_{q \in S^1}\left\|\frac{\partial \widetilde{\gamma}_s}{\partial s}(q)\right\| + \max_{q \in S^1}\left|\frac{\partial l}{\partial s}(s,q)\right|.
	\]
	Using the above equations one can show
	\[
		d_{B} (\alpha,\beta) \leq l_B(\hat{\gamma}_s) \leq C \cdot \left( \max_{q \in S^1} |f(q)| + \max_{q \in S^1}\left|\frac{\partial f}{\partial q} (q)\right| \right).
	\]
	Here the constant $C$ and the size of the $C^2$-neighbourhood of zero containing $f$ depend only on $\alpha$, more precisely on its curvature $k$.

\section{Discussion and open problems}\label{sec:discussion}

Theorem~\ref{thm:approx} provides a map of the space $\mathcal{AP}$ of marked polygons (up to orientation-preserving Euclidean isometries) into the completion of the set $\BM$ of symplectic diffeomorphisms with respect to the Hofer metric.

\begin{question}
	Is the map $\mathcal{AP} \to \overline{\BM}$ injective? 
\end{question}

\begin{question}
	Is the pullback of the Hofer metric to $\mathcal{A}\mathcal{P}$ given by semi-algebraic functions of the vertices? If not, does there exist an algorithm that computes this pullback with arbitrary precision?
\end{question}

Let $P$ be a marked polygon, and $\alpha_i$ be an approximating sequence from Theorem~\ref{thm:approx}. Consider the corresponding billiard ball maps $\psi_P$ (which is defined on a proper full-measure subset of the annulus) and $\psi_i := \psi_{\alpha_i}$.

Since $\psi_{i}$ is a Cauchy sequence with respect to the Hofer metric, the Lagrangian submanifolds $\graph(\psi_i) \subset (A \times A, \omega \oplus - \omega)$ form a Cauchy sequence with respect to the spectral metric (also called the $\gamma$-metric) on Lagrangian submanifolds of $A \times A$ (cf. \cite{viterbo2022supports}). For the elements in the completion of this space there is a notion of $\gamma$-support, which is a subset of $A \times A$. 

\begin{question}
	How is the $\gamma$-support of the limit of $\graph(\psi_i) \subset A \times A$ in the $\gamma$-metric related to the graph of $\psi_P$?
\end{question}

\begin{question}\label{q-dyn}
	What is the relation between the dynamics of $\psi_P$ and that of $\psi_i$ for $i \to \infty$?
\end{question}

Let us note that since $\psi_P$ is discontinuous, we cannot hope that $\psi_i$ converges to $\psi_P$ in the uniform norm. However,  the limit of $\graph(\psi_i)$ in the Hausdorff metric likely contains the graph of $\psi_P$.

A potential approach to Question \ref{q-dyn} could be as follows. It seems plausible that one can associate a version of filtered Floer or generating function homology with a billiard ball map $\psi$ or its iteration $\psi^k$, $k \in \mathbb{N}$,   cf. \cite{Cotton-Clay,Felshtyn}. This gives rise to
a persistence module along the lines of \cite{PS, PSRZ}. Write $\Barc(\psi^k)$ for the corresponding barcode.
Since the sequence $\psi_i^k$ is Cauchy with respect to the Hofer distance,
the corresponding sequence of barcodes $\Barc(\psi_i^k)$ should be Cauchy with respect to
the bottleneck distance (cf. \cite{PS,PSRZ}).
Hence it defines an element, say $\mathcal{C}_k$, in the completion of the space of barcodes. Now we are ready to formulate a more precise version of Question \ref{q-dyn}.

\begin{question}
\label{q-bar}
Which dynamical properties of the billiard in $P$ (length of periodic orbits, if any, their multiplicities, etc.) can be read from the sequence of barcodes $\mathcal{C}_k$?
\end{question}

Development of the Floer theory becomes more natural if one relates the group $\Symp(A, \partial A)$ to
the group $\Symp_c(A,\omega)$ of symplectic diffeomorphisms supported in the interior of $A$.

\begin{question}
		Can we isometrically embed  $\Symp(A, \partial A)$
into the completion of  $\Symp_c(A,\omega)$ with respect to the corresponding Hofer metric?
	\end{question}

Finally, Theorem~\ref{thm:approx} likely extends to more general convex (not necessarily
strictly convex) piecewise smooth curves other than polygons. That being said, the level of
generality in which the theorem holds is unclear to us.

\begin{question} Let $C$ be a continuous simple closed plane counterclockwise arc-length parametrized curve of length $1$ bounding
a convex set in $\mathbb{R}^2$. Does there exist a sequence of billiard tables $\alpha_i$ which is Cauchy in $d_B$-distance and which $C^0$-converges to $C$?
\end{question}

Let $\alpha$ be a billiard table. Fix $n$, and consider the {\it generating function} $F_{\alpha} \colon \mathbb{T}^n \to \mathbb{R}$ defined as
\[
		F_{\alpha}(q_1,\ldots, q_n) = \sum_{i=1}^n \|\alpha(q_{i+1}) - \alpha(q_i)\|,
\]
where $q_{n+1} = q_1$. It is well known (for example, see~\cite{katok1995introduction}) that every closed billiard trajectory of period $n$ is a critical point of the function $F_{\alpha}$. 
Now, consider a billiard table $\beta$ which is $C^0$-close to $\alpha$.
Let $F_{\beta} \colon \mathbb{T}^n \to \mathbb{R}$ be the corresponding generating function.  Then $F_{\beta}$ is $C^0$-close to $F_{\alpha}$. Indeed,
\begin{align*}
\sup |F_{\alpha}-F_{\beta}| &\leq \sup \sum_{i=1}^n \bigg|\|\alpha(q_{i+1}) - \alpha(q_i)\| - \|\beta(q_{i+1}) - \beta(q_i)\|  \bigg| \leq\\
&\leq \sup \sum_{i=1}^n \|(\alpha(q_{i+1}) - \alpha(q_i)) - (\beta(q_{i+1}) - \beta(q_i)) \| \leq \\
&\leq  \sup \sum_{i=1}^n \bigg[  \|\alpha(q_{i+1}) - \beta(q_{i+1})\|+ \|\alpha(q_{i}) - \beta(q_i) \|\bigg] \leq 2n \cdot \sup \|\alpha - \beta\|.
\end{align*}
An important source of the Morse-theoretic information about $F_{\alpha}$
(including the actions of $n$-periodic orbits and their indices) is its barcode. By the stability theorem (see, e.g., \cite{PSRZ}), the barcodes of $F_\alpha$ and $F_\beta$ are close in the bottleneck distance. Together with Lemma~\ref{lem:compare_bil_c0}, this observation suggests that our method is unlikely to provide additional information about the actions and indices of periodic orbits beyond that already available from generating functions.

In contrast to this, Corollary \ref{cor-dyn} above yields a dynamical phenomenon which apparently
cannot be obtained via generating functions, at least in a naive way. Let $y$ be an $n$-periodic point of the billiard ball map $\psi_{\alpha}$.  Consider any set $X \subseteq A$ with $e(X)>0$, such that $X$ and $\psi_\alpha^n(X)$ lie sufficiently close to $y$ with respect to the Euclidean metric on $A$. 
Let $\beta$ be another billiard table with $d_{B}(\alpha,\beta) < e(X)/4n$. By  Corollary \ref{cor-dyn}, there exists $x \in X$ with $\psi_\beta^n (x) \in \psi_\alpha^n(X)$. In other words, the trajectory of $x$ for time $n$ with respect to $\beta$ is ``almost" periodic:  the terminal point is close to the initial one. 

\begin{remark}
We conclude with a remark about the comparison of the two approaches. 
Suppose that a billiard table $\alpha$ has a periodic orbit of period $n$ and action $E$ corresponding to an endpoint of an interval in the barcode of the associated generating function. The stability theorem for barcodes should then imply that every sufficiently $C^0$-close billiard table $\beta$ also has a periodic orbit of period $n$ and action close to $E$\footnote{It would be interesting to complete the details of this argument.}.
However, we do not get any insight on the location of this orbit. At the same time, the displacement energy argument shows that if we consider any sufficiently $d_B$-close billiard table $\beta$, then near this periodic point of $\psi_\alpha$ there exists an ``almost" periodic point of $\psi_\beta$.
\end{remark}

\bibliography{bibliography}

\begin{thebibliography}{10}

\bibitem{arnaud2024higher}
M.-C. Arnaud, V.~Humili\`ere, and C.~Viterbo.
\newblock Higher dimensional {Birkhoff} attractors (with an appendix by
  {Maxime} {Zavidovique)}.
\newblock {\em Journal de l{\textquoteright}\'Ecole polytechnique {\textemdash}
  Math\'ematiques}, 13:759--808, 2026.

\bibitem{buhovsky2024dichotomy}
L.~Buhovsky, B.~Feuerstein, L.~Polterovich, and E.~Shelukhin.
\newblock A dichotomy for the {H}ofer growth of area preserving maps on the
  sphere via symmetrization.
\newblock 2024.
\newblock \href{https://arxiv.org/abs/2408.08854}{arXiv:2408.08854}.

\bibitem{chodosh2025p}
O.~Chodosh and S.~Cholsaipant.
\newblock The p-widths of a polygon.
\newblock 2025.
\newblock \href{https://arxiv.org/abs/2505.03047}{arXiv:2505.03047}.

\bibitem{Cotton-Clay}
A.~Cotton-Clay.
\newblock Symplectic {F}loer homology of area-preserving surface
  diffeomorphisms.
\newblock {\em Geometry \& Topology}, 13(5):2619--2674, 2009.

\bibitem{Felshtyn}
A.~Fel’shtyn.
\newblock Nielsen theory, {F}loer homology and a generalisation of the
  {P}oincare--{B}irkhoff theorem.
\newblock {\em Journal of Fixed Point Theory and Applications}, 3:191--214,
  2008.

\bibitem{gromov2007metric}
M.~Gromov.
\newblock {\em Metric structures for Riemannian and non-Riemannian spaces}.
\newblock Springer, 2007.

\bibitem{hofer_1990}
H.~Hofer.
\newblock On the topological properties of symplectic maps.
\newblock {\em Proceedings of the Royal Society of Edinburgh Section A:
  Mathematics}, 115(1-2):25–38, 1990.

\bibitem{humiliere2008some}
V.~Humili{\`e}re.
\newblock On some completions of the space of {H}amiltonian maps.
\newblock {\em Bulletin de la Soci{\'e}t{\'e} Math{\'e}matique de France},
  136(3):373--404, 2008.

\bibitem{katok1995introduction}
A.~Katok and B.~Hasselblatt.
\newblock {\em Introduction to the modern theory of dynamical systems}.
\newblock Number~54. Cambridge University Press, 1995.

\bibitem{lalonde1995geometry}
F.~Lalonde and D.~McDuff.
\newblock The geometry of symplectic energy.
\newblock {\em Annals of Mathematics}, 141(2):349--371, 1995.

\bibitem{Marvizi}
S.~Marvizi and R.~Melrose.
\newblock {Spectral invariants of convex planar regions}.
\newblock {\em Journal of Differential Geometry}, 17(3):475 -- 502, 1982.

\bibitem{Mcduff_salamon}
D.~McDuff and D.~Salamon.
\newblock {\em {Introduction to Symplectic Topology}}.
\newblock Oxford University Press, 03 2017.

\bibitem{polterovich2014function}
L.~Polterovich and D.~Rosen.
\newblock {\em Function theory on symplectic manifolds}, volume~34.
\newblock American Mathematical Soc., 2014.

\bibitem{PSRZ}
L.~Polterovich, D.~Rosen, K.~Samvelyan, and J.~Zhang.
\newblock {\em Topological persistence in geometry and analysis}, volume~74.
\newblock American Mathematical Soc., 2020.

\bibitem{PS}
L.~Polterovich and E.~Shelukhin.
\newblock Lagrangian configurations and {H}amiltonian maps.
\newblock {\em Compositio Mathematica}, 159(12):2483--2520, 2023.

\bibitem{Schneider_2013}
R.~Schneider.
\newblock {\em Convex Bodies: The Brunn–Minkowski Theory}.
\newblock Encyclopedia of Mathematics and its Applications. Cambridge
  University Press, 2nd edition, 2013.

\bibitem{siburg2004principle}
K.~F. Siburg.
\newblock {\em The principle of least action in geometry and dynamics}.
\newblock Number 1844. Springer Science \& Business Media, 2004.

\bibitem{tabachnikov2005geometry}
S.~Tabachnikov.
\newblock {\em Geometry and billiards}, volume~30.
\newblock American Mathematical Soc., 2005.

\bibitem{viterbo2022supports}
C.~Viterbo.
\newblock On the supports in the {H}umilière completion and
  $\gamma$-coisotropic sets.
\newblock {\em Compositio Mathematica}, 162(7):1439–1486, 2026.

\end{thebibliography}
\bibliographystyle{abbrv}
	
\end{document}